\documentclass[10pt]{article}

\usepackage[utf8]{inputenc}
\usepackage{mathtools,
        mathrsfs,amssymb,amsthm, amsmath,
        bm}
\usepackage{hyperref}
\usepackage{url}
\usepackage[a4paper, left = 1in, right = 1in, top =1in, bottom = 1in]{geometry} 
\usepackage{dsfont}   
\usepackage{enumitem}

\newtheorem{thm}{Theorem}
\newtheorem{pro}{Proposition}
\newtheorem{cor}
{Corollary}

\makeatletter
\renewenvironment{proof}[1][\proofname] {\par\pushQED{\qed}\normalfont\topsep6\p@\@plus6\p@\relax\trivlist\item[\hskip\labelsep\bfseries#1\@addpunct{.}]\ignorespaces}{\popQED\endtrivlist\@endpefalse}
\makeatother

\def\[#1\]{\begin{align*}#1\end{align*}}

\newcommand{\R}{\mathbb{R}}
\newcommand{\N}{\mathbb{N}}
\def\C{\mathbb{C}}
\newcommand{\ceq}{\coloneqq}
\renewcommand{\P}{\mathbb{P}}
\def\B{\mathscr{B}}
\newcommand{\lrtx}[1]{\ \text{#1} \ }
\newcommand{\M}
{\mathbb{M}}
\newcommand{\st}{\,\vert\,}
\makeatletter
\newcommand{\pr}{\ensuremath{'}}
\newcommand{\pfqed}{\tag*{\qedhere}}

\newcommand{\HD}{\mathbb{H}}
\newcommand{\diam}{\mathrm{diam}}

\begin{document}
\title{Some Remarks on Hausdorff Measurability of Lipschitz Images in Metric
Spaces}
\author{Yu-Lin Chou\thanks{Yu-Lin Chou, Institute of Statistics, National Tsing Hua University, Hsinchu 30013, Taiwan,  R.O.C.; Email: \protect\url{y.l.chou@gapp.nthu.edu.tw}.}}
\date{}
\maketitle

\begin{abstract}
In this short note, we show that, in any given metric space, every
Lipschitz open-map image of every subset of a given metric space whose
boundary is Hausdorff-null is Hausdorff-measurable with respect to
the same dimension. The main results are connected with a number of
familiar concepts in other branches such as complex analysis, functional
analysis, and topology.\\

\end{abstract}

\section{Introduction}

Hausdorff measures can be studied in metric spaces, which are apparently
quite more general than the usual Euclidean spaces, i.e. any $n$
copies of $\R$ equipped with the metric induced by the $l^{2}$-norm.
The investigation is done toward the general question: to what extent
one can obtain Hausdorff measurability results for the image of a
set under a map acting between metric spaces. This question would
be \textit{a priori} of intellectual interest; and its derivatives
could be potentially useful somewhere involving Hausdorff integrals,
and could be useful even from the viewpoint of applications external
to geometric measure theory and geometric analysis: fields such as
harmonic analysis (e.g. Mattila \cite{m}), probability theory (e.g.
Adler \cite{a} or Billingsley \cite{b1,b2} or Xiao \cite{x}), or
fractal geometry (e.g. Edgar \cite{e} or Falconer \cite{f}) serve
as significant examples.

Some preliminary remarks would be helpful for the purposes of communication.
Considering the nice analytic properties and the geometric meanings such as the notion of rectifiability associated with Lipschitz maps, 
in particular those maps acting between metric spaces whose “increment
ratios” are uniformly bounded, we take as our primary concern the
Lipschitz maps acting between metric spaces. Since Borel sigma-algebra
would be in many (practical) situations a natural, tractable choice
of measurability structure for metric spaces, and since Hausdorff
measures are always Borel (Theorem 27, Rogers \cite{r}), the continuity
of Lipschitz maps implies that every Lipschitz preimage of Borel sets
is Hausdorff-measurable; this fact is, in particular, invariant in
the choice of “dimension gauge” entering into the construction
of Hausdorff measures.

On the contrary, the Hausdorff measurability of Lipschitz images is
much less apparent even in a case involved only with the Euclidean
spaces; the reader is invited to consider the first assertion of Lemma 3.2 in Evans and Gariepy
\cite{eg}, which shows that the image of every Lebesgue-measurable
subset of $\R^{n}$ under any given Lipschitz map from $\R^{n}$ to
$\R^{m}$ is Hausdorff-measurable, where the reference Hausdorff measure
is the $n$-dimensional Hausdorff measure \[
A \mapsto \sup_{\delta > 0}\inf \bigg\{ \sum_{j \in \N}\frac{\pi^{n/2}}{\Gamma(\frac{n}{2} + 1)}\bigg( \frac{\diam(A_{j})}{2} \bigg)^{n} \,\bigg\vert\, \{ A_{j} \}_{j \in \N} \subset 2^{\R^{m}} \lrtx{covers} A,\ \sup_{j \in \N}\diam(A_{j}) \leq \delta \bigg\}
\]defined for all subsets $A$ of $\R^{m}$. The “covertness”
of Hausdorff measurability of Lipschitz images, as compared to the
apparentness of that of Lipschitz preimages, is arguably expected;
the present work is intended as a contribution to unravel the covertness.

Our main remarks, toward answering the question, are the following: 

\begin{thm}\label{main}
Let $\Omega_{1}, \Omega_{2}$ be metric spaces;
let $\HD_{1}^{\alpha}, \HD_{2}^{\alpha}$ be $\alpha$-dimensional Hausdorff measures over $\Omega_{1}$ and $\Omega_{2}$, respectively, for all $\alpha \in \R_{+}$.
If $f: \Omega_{1} \to \Omega_{2}$ is a Lipschitz, open map,
then $f^{[1]}(A) \in \sigma(\HD_{2}^{\alpha})$ for all $A \subset \Omega_{1}$ such that $\HD_{1}^{\alpha}(\partial A) = 0$ and for all $\alpha \in \R_{+}$.
\end{thm}

\begin{pro}\label{mmain}
Let $\Omega_{1}, \Omega_{2}$ be metric spaces; let $\HD_{1}^{\alpha}, \HD_{2}^{\alpha}$ be $\alpha$-dimensional Hausdorff measures over $\Omega_{1}$ and $\Omega_{2}$, respectively, for all $\alpha \in \R_{+}$; let $\M$ be a Radon measure over $\Omega_{1}$. If $\Omega_{1}$ is sigma-compact, if $f: \Omega_{1} \to \Omega_{2}$ is Lipschitz, and if $\HD_{1}^{\alpha} \leq \M$ on $\B_{\Omega_{1}}$, then $f^{[1]}(A) \in \sigma(\HD_{2}^{\alpha})$ for all $A \in \B_{\Omega_{1}}$.
\end{pro}As will be discussed, these results contain and are connected with
some familiar, important concepts and facts in different branches
of mathematics. The possibly to-be-clarified objects present in the
above statements will be immediately dealt with in the next section;
and the last section contains the major observations. 

\section{Preliminaries}

We denote by $\R_{+}$ the set of all reals $\geq 0$. Given any sets
$\Omega_{1}, \Omega_{2}$ and any function $f: \Omega_{1} \to \Omega_{2}$,
we denote by $f^{[1]}$ the corresponding image map $2^{\Omega_{1}} \to 2^{\Omega_{2}}$
that takes every $A \subset \Omega_{1}$ to be its $f$-image $\{ f(x) \st x \in A \}$.
If $\Omega$ is a metric space, the Borel sigma-algebra of $\Omega$
(with respect to the topology induced by the given metric) is denoted
$\B_{\Omega}$. For every Lipschitz function $f$, we denote by $|f|_{Lip}$
the Lipschitz constant of $f$. 

Throughout, given any set $\Omega$, by a \textit{measure \textup{(}over}
$\Omega$) we mean what is frequently called an outer measure (defined
on $2^{\Omega}$); this terminology is more convenient for our purposes
and has its roots in the spirit of Carath{\'e}odory's original developments.
Thus “Lebesgue measure” will mean what is usually referred to
as Lebesgue outer measure. If $\Omega$ is a set, and if $\M$ is
a measure on $2^{\Omega}$, we denote by $\sigma(\M)$ the collection
of all $\M$-measurable subsets of $\Omega$; the fact that such a
collection is always a sigma-algebra explains the notation. If $\Omega$
is a topological space, for a measure $\M$ over $\Omega$ to be a
\textit{Radon measure} we require, following Krantz and Parks \cite{k},
that $\sigma(\M) \supset \B_{\Omega}$, that $\M$ be finite at every
compact subset of $\Omega$, that $\M$ be inner regular at every
open subset of $\Omega$, and that $\M$ be outer regular at every
subset of $\Omega$. 

If $\Omega$ is a metric space, and if $\alpha \in \R_{+}$, by the
$\alpha$\textit{-dimensional Hausdorff measure over} $\Omega$ we
mean\footnote{This ``additional" clarification seems necessary as Hausdorff measures can be defined in a much more general way; one may be referred to Rogers \cite{r}.}
the measure \[
A \mapsto \sup_{\delta > 0}\inf \bigg\{ \sum_{j \in \N}(\diam(A_{j}))^{\alpha} \,\bigg\vert\, \{ A_{j} \}_{j \in \N} \subset 2^{\Omega} \lrtx{covers} A,\ \sup_{j \in \N}\diam(A_{j}) \leq \delta \bigg\}
\] on $2^{\Omega}$. Here the diameter function, $\diam$, is certainly
understood with respect to the given metric; we will not redundantly
make explicit this dependence. In general, we denote by $\HD^{\alpha}$
an $\alpha$-dimensional Hausdorff measure; whenever it is possible
for clarity to be compromised, the ambient space associated with $\HD^{\alpha}$
will be signified by, for instance, adding a subscript to $\HD^{\alpha}$.
Although the constant present in the definition of Hausdorff measures
over Euclidean spaces is absent from the definition of Hausdorff measures
over metric spaces, for our purposes this difference is immaterial
as we are not concerned with the exact size of certain objects.\footnote{If informative, the following would help fix an intuition behind Hausdorff measures. (More informative examples may be found in, for instance, Krantz and Parks \cite{k}.) The reader is invited to consider a line segment of unit length embedded into $\R^{n}$ where $n \geq 2$ is given. If the unit line segment is viewed as a subset of $\R$, then it has length $1$, which coincides with its  Lebesgue measure. However, for all $j \geq 2$, the (embedded) line segment in $\R^{j}$ has the corresponding Lebesgue measure $0$. Thus Lebesgue measure is not suitable as a modulus of size for embedded objects or ``lower-dimensional" objects in a given space, let alone as a device capturing   a notion of dimension for such objects. 

On the other hand, Hausdorff measures allow one to talk about size (and dimension) of embedded objects in an elegant and sensible way. If we return to consider the unit line segment embedded into $\R^{n}$ with $n \geq 2$ given, we may check out its $\alpha$-dimensional Hausdorff measures as $\alpha$ runs through $\R_{+}$, and see how the outcome agrees with intuitive expectation.  Indeed, if $\delta > j^{-1}$, then there are $j$ open balls in $\R^{n}$ of radius $1/2j$ such that these balls cover the line segment and have diameter uniformly less than $\delta$. Since the infimum appearing in the definition of Hausdorff measures is no greater than the sum $\sum_{l=1}^{j}j^{-\alpha} = j^{1-\alpha}$ of the $\alpha$-powered diameter of the balls, we see that $\alpha > 1$ implies that the $\alpha$-dimensional Hausdorff measure of the line segment is $=0$. A slightly further argument would show that the $\alpha$-dimensional Hausdorff measure of the line segment is $=1$ if $\alpha = 1$ and $=+\infty$ if $0 \leq \alpha < 1$. Here we have ignored the canonical constant usually entering into the definition of Hausdorff measures over a Euclidean space, which by chance does not matter. More importantly, the outcome agrees with intuition, even with a physical one. The outcome suggests, first off, that we get to measure the size of an embedded object ``correctly" if we stay in a ``right" zone, as captured by the choice of ``dimension gauge", having been denoted by $\alpha$, of the Hausdorff measure; however, taking a look at the object from ``too far away" obtains a null measure, and, from ``too close", an infinite measure. As a digression, the phenomena remind one of the global-local interpretation of  the concept of a manifold.}

For convenience, we will also refer to a cover such that the diameter
of the elements of the cover is uniformly small, say $\leq \delta$,
as a $\delta$-admissible cover. 

\section{Results}

Followed by two corollaries to it, Theorem \ref{main} is firstly
considered:

\begin{proof}[Proof of Theorem 1]

Let $\alpha \in \R_{+}$; let $A \subset \Omega_{1}$. Since $f$
is Lipschitz by assumption, given any $\delta > 0$ we have $\diam(f^{[1]}(A_{j})) \leq |f|_{Lip}\diam(A_{j}) \leq |f|_{Lip}\delta$
for all $j \in \N$ and all $\delta$-admissible covers $\{ A_{j} \}$
of $A$. If $\{ A_{j} \}$ is a $\delta$-admissible cover of $A$,
then $f^{[1]}(A) \subset \bigcup_{j}f^{[1]}(A_{j})$; so $\{ f^{[1]}(A_{j}) \}$
is a $|f|_{Lip}\delta$-admissible cover of $f^{[1]}(A)$. The diameter
inequalities above then imply \[
\HD_{2}^{\alpha}(f^{[1]}(A)) \leq |f|_{Lip}^{\alpha}\HD_{1}^{\alpha}(A).
\]

We have $A = A^{\circ} \cup (A\setminus A^{\circ})$, and $f^{[1]}(A) = f^{[1]}(A^{\circ}) \cup f^{[1]}(A \setminus A^{\circ})$.
Since $f$ is an open map by assumption, the image $f^{[1]}(A^{\circ})$
is an open subset and hence a Borel subset of $\Omega_{2}$. As every
Hausdorff measure is Borel, we have $f^{[1]}(A^{\circ}) \in \sigma(\HD^{\alpha}_{2})$.
If $\HD_{1}^{\alpha}(\partial A) = 0$, then \[
\HD_{2}^{\alpha}(f^{[1]}(A \setminus A^{\circ})) 
&\leq 
|f|_{Lip}^{\alpha}\HD_{1}^{\alpha}(A \setminus A^{\circ})\\ 
&\leq 
|f|_{Lip}^{\alpha}\HD_{1}^{\alpha}(\partial A)\\ 
&= 0.
\] Since, for every measure $\M$, the sigma-algebra $\sigma(\M)$ of
the $\M$-measurable sets is complete in the sense that every $\M$-null
set is $\M$-measurable, it follows that $f^{[1]}(A\setminus A^{\circ}) \in \sigma(\HD_{2}^{\alpha})$.
But then \[
f^{[1]}(A) \in \sigma(\HD_{2}^{\alpha}). \pfqed
\] 

\end{proof}

The requirement that $A$ have boundary of measure zero is to a certain
extent mild. On the other hand, the requirement is not artificial;
sets with null boundary play a significant role in the theory of weak
convergence of measures.\footnote{When probability measures are in particular of concern, a (measurable) subset of a metric space whose boundary is null with respect to a probability measure $\P$ is also called a  $\P$-\textit{continuity set}  (e.g. in Billingsley \cite{b}). In this regard, we might as well state the  conclusion of Theorem \ref{main} as an assertion that is valid for all $\HD_{2}^{\alpha}$-continuity sets, which would perhaps be more ``catchy".}

Theorem \ref{main} admits some interesting consequences that may
be worth being stated separately:

\begin{cor}

Let $\HD_{1}^{\alpha}, \HD_{2}^{\alpha}$ be as in Theorem \ref{main}.
If $\Omega_{1}, \Omega_{2}$ are Banach spaces, if $\Omega_{1}^{\pr} \subset \Omega_{1}$
is a  subspace, and if $f: \Omega_{1}^{\pr} \to \Omega_{2}$ is a
linear, bounded surjection, then there is some $g: \Omega_{1} \to \Omega_{2}$
such that $g$ satisfies the conclusion of Theorem  \ref{main}. 

\end{cor}

\begin{proof}

We have by the Hahn-Banach extension some linear $g: \Omega_{1} \to \Omega_{2}$
such that $g|_{\Omega_{1}^{\pr}} = f$ and $g$ has the same operator
norm as $f$. Since $f$ is surjective by assumption, it follows that
$g$ is surjective. If $|\cdot|_{1}, |\cdot|_{2}$ are the given norms
on $\Omega_{1}$ and $\Omega_{2}$, respectively, and if $|\cdot|_{op}$
denotes the operator norm, then, since $|g(x) - g(y)|_{2} = |g(x-y)|_{2} \leq |g|_{op}|x-y|_{1}$
for all $x,y \in \Omega_{1}$ by assumptions, the map $g$ is Lipschitz.
But, as $g$ is continuous, the open mapping theorem (in functional
analysis) implies that $g$ is an open map; the desired conclusion
then follows from Theorem \ref{main}.\end{proof}

\begin{cor}

Let $\HD_{1}^{\alpha}, \HD_{2}^{\alpha}$ be as in Theorem \ref{main}.

\begin{itemize}[leftmargin=*]
\item[(i)]
If all the assumptions of Theorem \ref{main} are the same but the assumption that $f: \Omega_{1} \to \Omega_{2}$ is a Lipschitz homeomorphism, then $f$ satisfies the conclusion of Theorem \ref{main};
\item[(ii)]
If $\Omega_{1} \subset \C$ is a region, if $\Omega_{2} \ceq \C$, and if $f: \Omega_{1} \to \Omega_{2}$ is non-constant, Lipschitz, and holomorphic, then $f$ satisfies the conclusion of Theorem \ref{main};
\item[(iii)]
If $\Omega_{1} \subset \R$ is an open interval, if $\Omega_{2} \ceq \R$, and if $f: \Omega_{1} \to \Omega_{2}$ is a differentiable open map with a bounded derivative, then $f$ satisfies the conclusion of Theorem \ref{main}.  
\end{itemize}

\end{cor}

\begin{proof}

The first case follows from the definition of a homeomorphism, which
ensures that $f$ is an open map.

For the second case, one may apply the open mapping theorem (in complex
analysis).

Once it is observed that the usual mean-value theorem in differential
calculus implies that $f$ is Lipschitz, the third case follows.\end{proof}

Proposition \ref{mmain} is a generalization of the proof idea of the first assertion of  Lemma 3.2 in Evans and Gariepy \cite{eg}, the essentials of which
can be adapted to a more general context:

\begin{proof}[Proof of Proposition \ref{mmain}]

To begin with, let $A \in \B_{\Omega_{1}}$ have finite $\M$-measure.
Since, by assumption, the measure $\M$ is Radon, the measure $\M$
is (e.g. Proposition 1.3.7, Krantz and Parks \cite{k}) inner regular
at every subset of $\Omega_{1}$ with finite $\M$-measure; so, in
particular, for every $j \in \N$ there is some compact $K_{j} \subset A$
such that $\M (K_{j}) > \M (A) - j^{-1}$. As $f$ is continuous,
which follows from the Lipschitz assumption, each of the images $f^{[1]}(K_{j})$
is compact and hence $\HD_{2}^{\alpha}$-measurable for every $\alpha \in \R_{+}$;
it follows from the union-preserving property of image maps that $f^{[1]}(\cup_{j}K_{j}) \in \sigma(\HD_{2}^{\alpha})$
for every $\alpha \in \R_{+}$. On the other hand, if $\HD_{1}^{\alpha} \leq \M$
on $\B_{\Omega_{1}}$, then \[
\HD_{2}^{\alpha}\bigg( f^{[1]}(A)\setminus f^{[1]}\bigg( \bigcup_{j}K_{j} \bigg) \bigg)
&\leq 
\HD_{2}^{\alpha}\bigg( f^{[1]}\bigg(A \setminus \bigcup_{j}K_{j} \bigg)\bigg)\\
&\leq 
|f|_{Lip}^{\alpha}\HD_{1}^{\alpha}\bigg(A \setminus \bigcup_{j}K_{j}\bigg)\\
&\leq 
|f|_{Lip}^{\alpha}\M \bigg( A \setminus \bigcup_{j}K_{j} \bigg).
\]Since $\M(A\setminus K_{j}) < j^{-1}$ for all $j \in \N$, the rightmost
term vanishes; it follows that $f^{[1]}(A) \in \sigma(\HD_{2}^{\alpha})$.

If $A \in \B_{\Omega_{1}}$ is of infinite $\M$-measure, the assumed
sigma-compactness of $\Omega_{1}$ implies that $A$ is a union of
elements of $\B_{\Omega_{1}}$ such that each of these Borel sets
is of finite $\M$-measure. That $f^{[1]}(A) \in \sigma(\HD_{2}^{\alpha})$
follows from the union-preserving property of image maps; this completes
the proof. \end{proof}

The assumptions of Proposition \ref{mmain} are not vacuous as, for
every $n \in \N$, the $n$-dimensional Hausdorff measure over $\R^{n}$
agrees (e.g. Theorem 30, Rogers \cite{r}) with the Lebesgue measure
over $\R^{n}$ modulo a constant multiple that depends at most on
$n$. In particular, when the constant $2^{-n}\pi^{n/2}/\Gamma(\frac{n}{2} + 1)$,
i.e. the volume of a unit-diameter $n$-ball, is employed to define
the $n$-dimensional Hausdorff measure over $\R^{n}$, the resultant
$n$-dimensional Hausdorff measure over $\R^{n}$ coincides exactly
(e.g. Theorem 2.5, Evans and Gariepy \cite{eg}) with the Lebesgue
measure over $\R^{n}$.

\end{document}